\documentclass[12pt]{article}

\usepackage{amssymb}%
\usepackage{amsmath}%
\usepackage{amsthm}%

\usepackage{fullpage}%
\usepackage{xcolor}%

\allowdisplaybreaks%

{\theoremstyle{plain}%
 \newtheorem{theorem}{Theorem}
 
 \newtheorem{conjecture}{Conjecture}
 \newtheorem{lemma}{Lemma}% 
}
{\theoremstyle{remark}

}
{\theoremstyle{definition}

}

\begin{document}

\begin{center}
{\Large Applications of the icosahedral equation for the Rogers--Ramanujan continued fraction} 
\end{center}

\begin{center}
{\textsc{John M. Campbell} }

 \ 

\end{center}

\begin{abstract}
 Let $R(q)$ denote the Rogers--Ramanujan continued fraction for $|q| < 1$. By applying the {\tt RootApproximant} command in the Wolfram language to 
 expressions involving the theta function $f(-q) := (q;q)_{\infty}$ given in modular relations due to Yi, this provides a systematic way of obtaining 
 experimentally discovered evaluations for $R\big(e^{-\pi\sqrt{r}}\big)$, for $r \in \mathbb{Q}_{> 0}$. We succeed in applying this approach to obtain 
 explicit closed forms, in terms of radicals over $\mathbb{Q}$, for the Rogers--Ramanujan continued fraction that have not previously been discovered 
 or proved. We prove our closed forms using the icosahedral equation for $R$ together with closed forms for and modular relations associated with 
 Ramanujan's $G$- and $g$-functions. An especially remarkable closed form that we introduce and prove is for $R\big( e^{-\pi \sqrt{48/5} } \big)$, in view 
 of the computational difficulties surrounding the application of an order-25 modular relation in the evaluation of $G_{48/5}$. 
\end{abstract}



\section{Introduction}
 Letting $|q| < 1$, the \emph{Rogers--Ramanujan continued fraction} is such that 
\begin{equation}\label{Rdefinition}
 R(q) = \frac{q^{1/5}}{1 + \cfrac{q}{1 + \cfrac{q^2}{1 + \cfrac{q^3}{1 + 
 \begin{matrix} \null \\ 
 \ddots 
 \end{matrix} } } } }, 
\end{equation}
 and one of the most fundamental properties concerning this continued fraction was introduced by Rogers \cite{Rogers189394} and is such that 
\begin{equation*}
 R(q) = q^{1/5} \prod_{n=1}^{\infty} \frac{ \left( 1 - q^{5n-1} \right) \left( 1 - q^{5n-4} \right) 
 }{ \left( 1- q^{5n-2} \right) \left( 1 - q^{5n-3} \right)}. 
\end{equation*}
 Among the most important identities for the function in \eqref{Rdefinition} are such that 
\begin{equation}\label{importanttheta1}
 \frac{1}{R(q)} - 1 - R(q) = \frac{ f\left( -q^{1/5} 
 \right) }{ q^{1/5} f\left( -q^{5} \right) } 
\end{equation}
 and 
\begin{equation}\label{importanttheta2}
 \frac{1}{R^{5}(q)} - 11 - R^{5}(q) = \frac{ f^{6}(-q) }{ q f^{6}\left( -q^{5} \right) }, 
\end{equation}
 writing $f(-q)$ in place of the the theta function that may be defined via either side of the formulation of Euler's pentagonal number theorem whereby 
\begin{equation*}
 \sum_{n = -\infty}^{\infty} (-1)^{n} q^{n(3n-1)/2} = (q;q)_{\infty}, 
\end{equation*}
 writing $(a;q)_{\infty} := \prod_{k=0}^{\infty} (1 - a q^k)$. As in \cite{ChernTang2021}, we express that there is a rich history associated with 
 modular relations involving the $R$-function and highlight the recursions whereby 
\begin{equation}\label{2777770727470717174717274757AM1A}
 \left(R\left(q^2\right) - R^2(q)\right) \left(1+R(q) R^2\left(q^2\right)\right) 
 = 2 R(q) R^3\left(q^2\right) 
\end{equation}
 and 
\begin{equation}\label{202222224202121272623292A2M2A}
 \left(R\left(q^3\right)-R(q)^3\right) \left(1+R(q) R\left(q^3\right)^3\right) 
 = 3 R(q)^2 R\left(q^3\right)^2. 
\end{equation}
 Of a similar importance, relative to the modular relations in \eqref{2777770727470717174717274757AM1A}
 and \eqref{202222224202121272623292A2M2A}, 
 is Ramanujan's formula 
\begin{equation}\label{202401121247twqelkveAM3A}
 R^5(q) 
 = R\left(q^5\right) 
 \frac{ 1-2 R\left(q^5\right)+4 R^2\left(q^5\right) - 
 3 R^3\left(q^5\right) + R^4\left(q^5\right) }{ 
 1+3 R\left(q^5\right)+4 R^2\left(q^5\right)+2 R^3\left(q^5\right) + R^4\left(q^5\right) } 
\end{equation}
 included in the first letter Ramanujan sent to Hardy, referring to the survey on \eqref{202401121247twqelkveAM3A} 
 its proofs given by Gugg \cite{Gugg2009}. 
 Recursions for the $R$-function as in \eqref{2777770727470717174717274757AM1A}--\eqref{202401121247twqelkveAM3A}
 typically are not sufficient in terms of the problem of determining an explicit closed form for 
 the Rogers--Ramanujan continued fraction, for a given argument. 
 In this paper, we introduce an experimental approach toward evaluating the Rogers--Ramanujan continued fraction, 
 relying on 
 the \emph{icosahedral equation} for $R$ \cite{Duke2005} 
 together with the experimental use of the {\tt RootApproximant} 
 command in the Wolfram language. We apply this approach 
 to obtain and prove new closed forms for $R$. 

 Our method also relies on closed forms for Ramanujan's class invariants
 together with modular relations for such invariants, 
 and may be applied broadly to produce new evaluations for $R$. 
 This approach 
 may also be applied 
 to produce new and simplified proofs of previously published 
 evaluations for $R$. 
 This is inspired by what Berndt et al.\ \cite{BerndtChanZhang1996} 
 describe as Ramanathan's \emph{uniform} 
 approach \cite{Ramanathan1987} toward the evaluation of the Rogers--Ramanujan continued fraction. 
 Ramanathan's method relied on class groups for imaginary quadratic fields and Kronecker's limit formula, 
 and Ramanathan applied this method to evaluate $R\big( e^{- \sqrt{r} \pi} \big)$ 
 for a number of rational values $r$ \cite{Ramanathan1987}. 

 There is a rich history surrounding the evaluation of 
\begin{equation}\label{richhistory}
 R\left(e^{-\pi \sqrt{\frac{n}{d}}}\right)
\end{equation}
 for a positive rational $\frac{n}{d}$ such that the denominator $d$, 
 with the fraction $\frac{n}{d}$ written in simplest form, 
 is divisible by $5$. 
 Notable surveys on such evaluations are given in \cite{Paek2021,YiPaek2022}. 
 With regard to these surveys \cite{Paek2021,YiPaek2022} 
 and the references therein, 
 evaluations for \eqref{richhistory}, 
 for denominators of the specified form, 
 are known for fractions $\frac{n}{d}$ equal 
 to $ \frac{1}{5}$, $ \frac{2}{5}$, $ \frac{3}{5}$, $ \frac{4}{5}$, 
 $\frac{6}{5}$, $ \frac{8}{5}$, $ \frac{9}{5}$, $ \frac{12}{5}$, 
 $ \frac{14}{5}$, $ \frac{16}{5}$, $ \frac{18}{5}$, $ \frac{24}{ 5}$, $ \frac{28}{5}$, $ \frac{32}{5}$, $ \frac{36}{5}$, 
 $\frac{44}{5}$, $ \frac{52}{5}$, $\frac{56}{5}$, $\frac{64}{5}$, $ \frac{68}{5}$, $ \frac{72}{5}$, $\frac{116}{5}$, 
 $\frac{164}{5}$, $ \frac{212}{5}$, $ \frac{404}{5}$, $ \frac{1}{10}$, $ \frac{1}{15}$, $ \frac{2}{15}$, $ \frac{4}{15}$, 
 $ \frac{8}{15}$, $ \frac{3}{20}$, $ \frac{9}{20}$, $ \frac{2}{25}$, $ \frac{4}{25}$, $ \frac{8}{25}$, $ \frac{12}{25}$, 
 $ \frac{2}{35}$, $ \frac{4}{35}$, $\frac{8}{35}$, $ \frac{1}{45}$, $\frac{2}{45}$, $ \frac{4}{45}$, $ \frac{8}{45}$, 
 $\frac{4}{55}$, $ \frac{1}{60}$, $ \frac{4}{65}$, $ \frac{4}{75}$, $ \frac{3}{80}$, $ \frac{4}{85}$, $ \frac{4}{145}$, 
 $ \frac{1}{180}$, $ \frac{4}{205}$, $ \frac{1}{240}$, $ \frac{4}{265}$, $ \frac{4}{445}$, and $ \frac{4}{505}$. 
 We succeed in proving new closed forms for \eqref{richhistory}, 
 for values $\frac{n}{d}$ not recorded in surveys such as 
 \cite{Paek2021,YiPaek2022} on closed forms for the Rogers--Ramanujan continued fraction. In particular, 
 we obtain explicit evaluations for each of the following expressions: 
 $$ R\left( e^{-\pi \sqrt{\frac{26}{5}}} \right), \ \ 
 R\left( e^{-\pi \sqrt{\frac{38}{5}}} \right), \ \ \text{and} \ \ 
 R\left( e^{-\pi \sqrt{\frac{48}{5}}} \right). $$
 We conclude with a conjectural closed form for 
 $ R\big( e^{-\pi \sqrt{16/15}} \big)$. 

\section{Preliminaries and background}
 The \emph{elliptic lambda function} satisfies 
\begin{equation}\label{202401307775717A7M71A}
 \lambda^{\ast}(r) = \frac{ \theta_{2}^{2}\left( 0, e^{-\pi \sqrt{r}} \right) }{ 
 \theta_{3}^{2}\left( 0, e^{-\pi \sqrt{r}} \right) }, 
\end{equation}
 where the Jacobi theta functions $\theta_{2}$ and $\theta_{3}$ are such that 
\begin{equation}\label{202401306443MA1A}
 \theta_{2}(z, q) = 2 q^{1/4} \sum_{n=0}^{\infty} 
 q^{n(n+1)} \cos((2n+1)z) 
\end{equation}
 and 
\begin{equation}\label{theta3definition}
 \theta_{3}(z, q) = 1 + 2 \sum_{n=1}^{\infty} q^{n^2} \cos(2 n z). 
\end{equation}
 The special function in \eqref{202401307775717A7M71A} is closely related to 
 Ramanujan's $G$- and $g$-functions, which may, respectively, be defined so that 
 $$ G_{n} = \frac{1}{ 2^{1/4} e^{-\pi \sqrt{n}/24} } \prod_{j=0}^{\infty} 
 \left( 1 + e^{-(2j+1) \pi \sqrt{n}} \right) $$
 and 
 $$ g_{n} = \frac{1}{2^{1/4} e^{-\pi \sqrt{n}/24} } \prod_{j=0}^{\infty} 
 \left( 1 - e^{-(2j+1) \pi \sqrt{n}} \right). $$
 Explicitly, we have that 
\begin{equation}\label{202480828083878387A8M1A}
 \lambda^{\ast}(n) = \frac{1}{2} \left( \sqrt{1 + G_{n}^{-12}} - \sqrt{1 - G_{n}^{-12}} \right) 
\end{equation}
 and that $$ \lambda^{\ast}(n) = g_{n}^{6} 
 \left( \sqrt{g_{n}^{12} + g_{n}^{-12}} - g_{n}^{6} \right). $$
 
 The Jacobi theta functions in \eqref{202401306443MA1A} and \eqref{theta3definition} 
 play a key role in the below formulation of the 
 \emph{icosahedral equation}, which is of key importance in our work. 
 The icosahedral equation for the Rogers--Ramanujan continued fraction \cite{Duke2005} 
 may be formulated in such a way so that 
\begin{equation}\label{icosahedral}
 1728 J(\tau) = -\frac{\left(r^{20}-228 r^{15}+494 r^{10}+228 r^5+1\right)^3}{r^5 \left(r^{10}+11 r^5-1\right)^5}, 
\end{equation}
 where $J$ denotes \emph{Klein's absolute invariant}, which may be defined as follows, 
 and where $r = r(\tau) = R\big( e^{2 \pi i \tau} \big)$. 
 By analogy with \eqref{202401307775717A7M71A}, we write 
 $$ \lambda(\tau) = \frac{ \theta_{2}^{4}(e^{i \pi \tau}) }{ \theta_{3}^{4}(e^{i \pi \tau}) }. $$
 The $J$-function may then be defined so that 
 $$ J(\tau) = \frac{4}{27} \frac{ \left(1- \lambda(\tau)+\lambda^2(\tau)\right)^3}{ \lambda^2(\tau)
 (1 - \lambda(\tau))^2 }. $$ 

 Duke \cite{Duke2005} proved that if $\tau$ is in an imaginary quadratic field, then $r(\tau)$ can be expressed in terms of radicals over $\mathbb{Q}$. 
 However, it is not enough to have a closed form for $J(\tau)$ to express $r(\tau)$ with nested radicals over $\mathbb{Q}$. Typically, it is very difficult 
 to obtain such expressions for the Rogers--Ramanujan continued fraction, as evidenced by the unknown values for \eqref{richhistory} suggested by the 
 surveys referenced above \cite{Paek2021,YiPaek2022}. Since the icosahedral equation in \eqref{icosahedral} provides a degree-60 polynomial identity 
 for $r(\tau)$ with algebraic coefficients, it is typically not feasible to simply 
 solve for $r(\tau)$ according to this polynomial identity. 
 Furthermore, computing $\lambda(\tau)$ for rational values $\tau$ is generally nontrivial, 
 and state-of-the-art Computer Algebra Systems such as Mathematica are generally not capable of 
 solving for required values involved in higher-order modular equations, 
 as in the order-25 modular relation shown below. 

 One of the most basic modular relations involving Ramanujan's $g$- and $G$-functions is such that 
\begin{equation}\label{20240203851AM2A}
 g_{4n} = 2^{1/4} g_{n} G_{n} 
\end{equation}
 for $n > 0$ \cite[p.\ 187]{Berndt1998}. 
 One of the keys to all of our proofs 
 is given by the modular relation given as follows and appearing in an equivalent form 
 in Part V of Ramanujan's Notebooks \cite[p.\ 222]{Berndt1998}. If 
\begin{equation}\label{alphafor25}
 \alpha = \frac{1}{2} \left( 1 - \frac{\sqrt{G_{n}^{24} - 1}}{G_{n}^{12}} \right) 
\end{equation}
 and 
\begin{equation}\label{betafor25}
 \beta = \frac{1}{2} \left( 1 - \frac{\sqrt{G^{24}_{n/25} - 1} }{G^{12}_{n/25}} \right)
\end{equation}
 and 
\begin{equation}\label{202747017277797117271P7M1787887A}
 P = \sqrt[12]{16 \alpha \beta (1-\alpha ) (1-\beta )} 
\end{equation}
 and 
\begin{equation}\label{Qfor25}
 Q = \sqrt[8]{\frac{\beta (1-\beta )}{\alpha (1-\alpha )}}, 
\end{equation}
 then 
\begin{equation}\label{modular25}
 Q+\frac{1}{Q} = -2 \left(P-\frac{1}{P}\right). 
\end{equation}

 Our experimental approach toward deriving closed forms for $R$ from the Theorem due to Yi given below relies on the application of the {\tt 
 RootApproximant} command to expressions of the form shown in \eqref{20240130874777A77M1A}. For example, for the $n = \frac{13}{2}$ case of 
 Theorem \ref{Yimaintheorem}, our experimental use of the {\tt RootApproximant} command in the Wolfram Language has led us to conjecture that 
\begin{equation}\label{20324012571622p2pqp2}
 s\left( \frac{13}{2} \right) = -37296+16705 \sqrt{5}+2 \sqrt{65 \left(10716449-4792536 \sqrt{5}\right)}, 
\end{equation}
 and this is proved in an equivalent way in Section \ref{20240131852AM1A}. 
 It seems that this approach toward the experimental application of Yi's results \cite{Yi2001Evaluations} 
 has not been considered 
 previously in publications related to \cite{Yi2001Evaluations}, 
 including \cite{BaruahSaikia2007,Yi2001construction,Yi2004,YiLeePaek2006}. 
 The following result due to Yi can be seen to be equivalent to 
 the previously known modular relation shown in \eqref{importanttheta2}. 

\begin{theorem}\label{Yimaintheorem}
 (Yi, 2001) Setting $q = q_{n} = e^{-2 \pi \sqrt{n/5}}$ and 
\begin{equation}\label{20240130874777A77M1A}
 s_{n} = \frac{ f^{6}(-q) }{ 5 \sqrt{5} q f^{6}(-q^5) }, 
\end{equation}
 we have that 
 $$ R^{5}\left( e^{-2 \pi \sqrt{n/5}} \right) = \sqrt{a^2 + 1} - a, $$
 for $2 a = 5 \sqrt{5} s_{n} + 11$ \cite{Yi2001Evaluations}. 
\end{theorem} 

    A remarkable aspect about how we have experimentally discovered closed forms from     Theorem \ref{Yimaintheorem} is given by how      the companion  
   relation in \eqref{importanttheta1} to \eqref{importanttheta2}    cannot be applied similarly via the {\tt RootApproximant} command.   
  To prove our results obtained from experimentally discovered closed forms for 
 \eqref{20240130874777A77M1A}, we rely on 
 the icosahedral equation. 
 It seems that this approach has not been considered in 
 past work related to Duke's work on the icosahedral equation \cite{Duke2005}, including 
 \cite{AndrewsBerndt2018,Berndt2006,Maier2009}. 

 A notable reference concerning special values of the Rogers--Ramanujan continued fraction is due to Berndt and Chan \cite{BerndtChan1995}. 
 Following Berndt and Chan's work in \cite{BerndtChan1995}, we recall the formulas
 $$ R\left( e^{-2\pi} \right) 
 = \sqrt{\frac{5 + \sqrt{5}}{2}} - \frac{\sqrt{5} + 1}{2} $$
 and 
 $$ R\left( e^{-2\pi \sqrt{5}} \right) = \frac{\sqrt{5}}{1 + \left( 
 5^{3/4} \left( \frac{\sqrt{5} - 1}{2} \right)^{5/2} - 1 \right)^{1/5}} - \frac{\sqrt{5} + 1}{2} $$
 given, respectively, in Ramanujan's first and second letters to Hardy. 
 Berndt and Chan \cite{BerndtChan1995} proved a number of 
 evaluations for $R$ that were discovered by Ramanujan 
 and that were not proved prior to \cite{BerndtChan1995}, 
 including the evaluation 
 $$ R\left( e^{-8\pi} \right) = \sqrt{a^2 + 1} - a, $$ where 
 $$ a = \frac{1}{2} \left(1+\frac{\sqrt{5} \left(3+\sqrt{2}+\sqrt{2} \sqrt[4]{5}-\sqrt{5}\right)}{3+\sqrt{2}-\sqrt{2} \sqrt[4]{5}-\sqrt{5}}\right). $$ 
 Our new results as in Theorem \ref{theorem38over5} 
 below are inspired by \cite{BerndtChan1995} 
 and are denoted by analogy with the main Theorems in \cite{BerndtChan1995}. 

\section{A new closed form related to $g_{130}$}\label{20240131852AM1A}
 Our discovery and our proof of the new result highlighted in Theorem \ref{theoremg130} 
 below require a ``twofold'' application of the {\tt RootApproximant} command, 
 in the sense that this command was the key to our experimental discovery of the conjectured value of 
 $s\big( \frac{13}{2} \big)$ shown in \eqref{20324012571622p2pqp2}
 and that the {\tt RootApproximant} function was required in our experimental discovery of the 
 value of $\lambda^{\ast}\big( \frac{26}{5} \big)$ 
 that is of key importance in our proof of Theorem \ref{theoremg130}. 
 It is not possible to derive this value from that of $g_{130}$ or $G_{130}$ using 
 the order-25 modular relation in \eqref{modular25}, 
 in the sense that current Computer Algebra Systems cannot solve for the 
 desired value of $G_{26/5}$. 
 
\begin{theorem}\label{theoremg130}
 The evaluation 
 $$ R^{5}\left( e^{-\pi \sqrt{\frac{26}{5}}} \right) 
 = \sqrt{a^2+1}-a$$
 holds, where 
\begin{equation}\label{2808284808183880818181808A8M1A}
 a = 208818-93240 \sqrt{5}-57825 \sqrt{13}+25900 \sqrt{65}. 
\end{equation} 
\end{theorem}

\begin{proof}
 According to the relation whereby 
 $$ G_{n} = \left( \frac{1}{2} \left( g_{n}^{8} + \frac{ \sqrt{g_{n}^{8} \left( 1 + g_{n}^{24} 
 \right)} }{ g_{n}^{8} } \right) \right)^{1/8}, $$
 we obtain a closed form for $G_{130}$, according to the closed form 
 $$ g_{130} = \left( \frac{\sqrt{5} + 1}{2} \right)^{3/2} \left( \frac{\sqrt{13} + 3}{2} \right)^{1/2} $$
 given in Ramanujan's Notebooks \cite[p.\ 203]{Berndt1998}. 
 Explicitly, $G_{130}$
 is given by the following Mathematica input. 
\begin{verbatim}
((161 + 72*Sqrt[5])*(119 + 33*Sqrt[13]) + 6*Sqrt[40779771 + 
18236316*Sqrt[5] + 11309683*Sqrt[13] + 5058108*Sqrt[65]])^(1/8)/2^(1/4)
\end{verbatim}
 We claim that $\lambda^{\ast}\big( \frac{26}{5} \big)$ is equal to 
 the sixth root, according to Mathematica's ordering for polynomial roots, of the following polynomial: 
\begin{align}
\begin{split}
 & x^8+14999688 x^7+140280340 x^6+14999688 x^5-280560666 x^4-14999688 x^3 + \\
 & 140280340 x^2-14999688 x+1. 
\end{split}\label{202401291115P899998888877777777M71A}
\end{align}
 According to the relationship in \eqref{202480828083878387A8M1A} 
 between $\lambda^{\ast}$ and Ramanujan's $G$-function, we have that 
\begin{equation}\label{20240129111088888888P88M1A}
 G_{26/5} = 2^{-1/12} \left( \left( \lambda^{\ast}\left( \frac{26}{5} \right) \right)^{2} - 
 \left( \lambda^{\ast}\left( \frac{26}{5} \right) \right)^{4} \right)^{-1/24}. 
\end{equation}
 First, we set $\alpha$ as in \eqref{alphafor25} for $n = 130$, 
 and we then express $\alpha$ in closed form, according to the above closed form for $G_{130}$.
 We then set $b$ as 
\begin{equation}\label{20240129110777777777q7q797q7P7qMq1qA}
 b = \frac{1}{2} \left( 1 - \frac{\sqrt{\mathcal{G}^{24} - 1} }{\mathcal{G}^{12}} \right), 
\end{equation}
 where, by analogy with \eqref{20240129111088888888P88M1A}, we set 
 \begin{equation}\label{20240129111088888888P88M1A}
 \mathcal{G} = 2^{-1/12} \left( \mathcal{L}^{2} - 
 \mathcal{L}^{4} \right)^{-1/24}, 
\end{equation}
 where $\mathcal{L}$ denotes the sixth root of the polynomial in \eqref{202401291115P899998888877777777M71A}. 
 By analogy with \eqref{202747017277797117271P7M1787887A} and \eqref{Qfor25},
 we set $$ P = \sqrt[12]{16 \alpha b (1-\alpha ) (1-b )} $$ and 
 $$ Q = \sqrt[8]{\frac{b (1-b )}{\alpha (1-\alpha )}}. $$ 
 It is thus a matter of routine to verify that the minimal polynomial of 
\begin{equation}\label{202401299199193919PM1A}
 Q+\frac{1}{Q} - \left( -2 \left(P-\frac{1}{P}\right) \right) 
\end{equation}
 with respect to a variable $x$ is $x$, i.e., so that 
 \eqref{202401299199193919PM1A} vanishes. 
 This may be verified with the output corresponding to the following Mathematica input. 
\begin{verbatim}
\[Alpha] = (1 - (8*Sqrt[-1 + ((161 + 72*Sqrt[5])*(119 + 33*Sqrt[13]) + 
6*Sqrt[40779771 + 18236316*Sqrt[5] + 11309683*Sqrt[13] + 
5058108*Sqrt[65]])^3/64])/((161 + 72*Sqrt[5])*(119 + 33*Sqrt[13]) + 
6*Sqrt[40779771 + 18236316*Sqrt[5] + 11309683*Sqrt[13] + 
5058108*Sqrt[65]])^(3/2))/2; 
L = Root[1 - 14999688*#1 + 140280340*#1^2 - 14999688*#1^3 - 
280560666*#1^4 + 14999688*#1^5 + 140280340*#1^6 + 14999688*#1^7 + 
#1^8 & , 6, 0]; 
G = 1/(2^(1/12)*(L^2 - L^4)^(1/24));
b = (1 - Sqrt[-1 + G^24]/G^12)/2;
P = 2^(1/3)*((1 - b)*b*(1 - \[Alpha])*\[Alpha])^(1/12);
Q = (((1 - b)*b)/((1 - \[Alpha])*\[Alpha]))^(1/8); 
MinimalPolynomial[Q + 1/Q - (-2 (P - 1/P)), x]
\end{verbatim}
 So, we have shown that: If we write $G_{26/5}$ in terms of $\lambda^{\ast}\big( \frac{26}{5} \big)$, 
 and then replace $\lambda^{\ast}\big( \frac{26}{5} \big)$ with $\mathcal{L}$, 
 and then redefine $\beta$ in \eqref{betafor25} according to this replacement, 
 the desired modular relation in \eqref{modular25} holds. 

 If we again write $G_{26/6}$ in terms of $\lambda^{\ast}\big( \frac{26}{5} \big)$, 
 but then replace $\lambda^{\ast}\big( \frac{26}{5} \big)$ with a free parameter $\ell$, 
 and again redefine $\beta$ in \eqref{betafor25} according to this replacement, 
 and similarly for $P$ and $Q$, if we then consider 
 the expression in \eqref{202401299199193919PM1A} as a function of $\ell$, we find that 
 there are two positive real roots of this function. 
 So, we can conclude that $\mathcal{L}$ is one of these two values. 
 By considering the possible numerical values of the elliptic lambda function on its domain, 
 we can conclude that $\mathcal{L} = \lambda^{\ast}\big( \frac{26}{5} \big)$. 

 From the relationship between $\lambda^{\ast}$ and Ramanujan's $g$-function, we have that 
 $$ \frac{\sqrt[12]{\frac{1}{\mathcal{L}} - \mathcal{L}}}{\sqrt[3]{2}} = g_{13/10} G_{13/10}. $$
 This, in turn, yields 
 $$ \frac{\sqrt[12]{\frac{1}{\mathcal{L}} - \mathcal{L}}}{\sqrt[3]{2}} 
 = \frac{\sqrt[12]{\frac{1}{\lambda^{\ast}\left(\frac{13}{10}\right)}-\lambda^{\ast}\left(\frac{13}{10}\right)}}{\sqrt[6]{2}
 \sqrt[24]{ \left( \lambda^{\ast}\left(\frac{13}{10}\right) \right)^2 - \left( \lambda^{\ast}\left(\frac{13}{10}\right) \right)^4}}. $$
 Writing 
 $$ \mathcal{C} = \frac{1}{16} \left(\frac{1}{\mathcal{L}} - \mathcal{L} \right)^2, $$
 we thus obtain that 
\begin{equation}\label{202402031114AM717777A}
 \lambda^{\ast}\left( \frac{13}{10} \right) 
 = \frac{\sqrt{\frac{\sqrt{4 \mathcal{C} + 1}}{\mathcal{C}}-\frac{1}{\mathcal{C}}}}{\sqrt{2}}. 
\end{equation}
 By setting $\tau = i \sqrt{\frac{13}{10}}$ in the icosahedral equation in \eqref{icosahedral}, 
 we obtain a polynomial involving 
 $R\big( e^{-\pi \sqrt{ 26/5 }} \big)$ 
 and $\lambda\big( i \sqrt{\frac{13}{10}} \big)$. 
 Equivalently, we obtain a polynomial involving 
 $R\big( e^{-\pi \sqrt{ 26/5 }} \big)$ and $\lambda^{\ast}\big( \frac{13}{10} \big)$: 
\begin{align}
\begin{split}
 & \left(r^{20}-228 r^{15}+494 r^{10}+228 r^5+1\right)^3 \left( \lambda^{\ast}\left(\frac{13}{10}\right) \right)^4 
 \left(1 - \left( \lambda^{\ast}\left(\frac{13}{10}\right) \right)^2\right)^2 \\ 
 & = 256 r^5 \left(r^{10}+11 r^5-1\right)^5 
 \left( \left( \lambda^{\ast}\left(\frac{13}{10}\right) \right)^2 - \left( \lambda^{\ast}\left(\frac{13}{10}\right) \right)^4-1\right)^3, 
\end{split}\label{20240130111929A9M9A}
\end{align}
 where $r = R\big( e^{-\pi \sqrt{ 26/5 }} \big)$. 
 We claim that: By replacing $r$ in \eqref{20240130111929A9M9A} with 
 $\sqrt[5]{ \sqrt{a^2+1}-a}$, for $a$ as in \eqref{2808284808183880818181808A8M1A}, 
 the same equality in \eqref{20240130111929A9M9A} holds. 
 For this latter expression for $r$, it is a matter of routine 
 to verify that the left-hand side of \eqref{20240130111929A9M9A} minus the right-hand side 
 of \eqref{20240130111929A9M9A} vanishes, by considering the minimal polynomial for this difference. 
 This may be verified with the output corresponding to the following Mathematica input. 
\begin{verbatim}
a = 208818 - 93240 Sqrt[5] - 57825 Sqrt[13] + 25900 Sqrt[65];
L = Root[1 - 14999688*#1 + 140280340*#1^2 - 14999688*#1^3 - 
280560666*#1^4 + 14999688*#1^5 + 140280340*#1^6 + 14999688*#1^7 + 
#1^8 &, 6, 0];
c = 1/16 (1/L - L)^2;
tau = I Sqrt[13/10];
r = (Sqrt[a^2 + 1] - a)^(1/5);
lambda = Sqrt[-(1/c) + Sqrt[1 + 4 c]/c]/Sqrt[2];
MinimalPolynomial[(r^20 - 228 r^15 + 494 r^10 + 228 r^5 + 1)^3 lambda^4 
(1 - lambda^2)^2 - 256 r^5 (r^10 + 11 r^5 - 1)^5 (lambda^2 - lambda^4 - 
1)^3 , x]
\end{verbatim}
 So, we have shown that $R\big( e^{-\pi \sqrt{ 26/5 }} \big)$ and 
 $\sqrt[5]{ \sqrt{a^2+1}-a}$ are roots of the polynomial indicated above. 
 By considering all possible numerical values of the roots of this polynomial, 
 we can conclude that $R\big( e^{-\pi \sqrt{ 26/5 }} \big) = \sqrt[5]{ \sqrt{a^2+1}-a}$. 
\end{proof}

\section{A new closed form related to $g_{190}$}\label{20240131852AM1A}
 Again, we require a twofold application of {\tt RootApproximant} 
 in regard to our closed form for \eqref{20240204545404A4M4A}, 
 in the sense that the {\tt RootApproximant} command was of key importance 
 in our experimental discovery of both the closed form for 
 \eqref{20240204545404A4M4A} and of the required value for the 
 elliptic lambda function involved in our proof. 

\begin{theorem}\label{theorem38over5}
 The evaluation 
\begin{equation}\label{20240204545404A4M4A}
 R^{5}\left( e^{-\pi \sqrt{\frac{38}{5}}} \right) = \sqrt{a^2+1} - a 
\end{equation}
 holds, where $$ a = 4165218 + 2945250 \sqrt{2} - 1862095 \sqrt{5} - 1316700 \sqrt{10}. $$
\end{theorem}

\begin{proof}
 Using the closed form 
\begin{equation}\label{g190closed}
 g_{190} = \left( \frac{1+\sqrt{5}}{2} \right)^{3/2} \left( 3 + \sqrt{10} \right)^{1/2} 
\end{equation}
 given in Part V of Ramanujan's Notebooks \cite[p.\ 203]{Berndt1998}, 
 we can show, as below, 
 that $\lambda^{\ast}\big( \frac{38}{5} \big)$ is equal to the sixth root of 
\begin{align}
\begin{split}
 & 1 - 632783448 x + 12127295380 x^2 - 632783448 x^3 - 24254590746 x^4 + \\
 & 632783448 x^5 + 12127295380 x^6 + 632783448 x^7 + x^8. 
\end{split}\label{20240203111868A8871M}
\end{align}
 From the closed form for $g_{190}$ in \eqref{g190closed}, 
 we find that $G_{190}$ is equal to the expression given by the following Mathematica input. 
\begin{verbatim}
(2/(((1 + Sqrt[5])^12*(3 + Sqrt[10])^4)/4096 + Sqrt[4096/((1 + 
Sqrt[5])^12*(3 + Sqrt[10])^4) + ((1 + Sqrt[5])^24*(3 + 
Sqrt[10])^8)/16777216]))^(-1/8)
\end{verbatim}
 Setting $n = 190$ in the order-25 modular relations
 among \eqref{alphafor25}--\eqref{modular25}, 
 we obtain a closed form for $\alpha$ from the closed form for $G_{190}$. 
 Using a similar argument as in the proof of Theorem \ref{theoremg130}, 
 we can show how the Mathematica output associated with the following input 
 proves the desired evaluation for $\lambda^{\ast}\big( \frac{38}{5} \big)$. 
\begin{verbatim}
\[Alpha] = (1 - (2*Sqrt[2*(-1 + (((1 + Sqrt[5])^12*(3 + Sqrt[10])^4)/4096 + 
Sqrt[4096/((1 + Sqrt[5])^12*(3 + Sqrt[10])^4) + ((1 + Sqrt[5])^24*(3 + 
Sqrt[10])^8)/16777216])^3/8)])/(((1 + Sqrt[5])^12*(3 + Sqrt[10])^4)/4096 + 
Sqrt[4096/((1 + Sqrt[5])^12*(3 + Sqrt[10])^4) + ((1 + Sqrt[5])^24*(3 + 
Sqrt[10])^8)/16777216])^(3/2))/2; 
L = Root[1 - 632783448*#1 + 12127295380*#1^2 - 632783448*#1^3 - 
24254590746*#1^4 + 632783448*#1^5 + 12127295380*#1^6 + 632783448*#1^7 + 
#1^8 & , 6, 0];
G =2^(-1/12) (L^2 - L^4)^(-1/24);
\[Beta] = (1 - Sqrt[-1 + G^24]/G^12)/2;
P = 2^(1/3)*((1 - \[Alpha])*\[Alpha]*(1 - \[Beta])*\[Beta])^(1/12);
Q = (((1 - \[Beta])*\[Beta])/((1 - \[Alpha])*\[Alpha]))^(1/8);
MinimalPolynomial[Q + 1/Q - (-2 (P - 1/P)), x]
\end{verbatim}
 By mimicking the derivation of \eqref{202402031114AM717777A}, this 
 can be used to express $\lambda^{\ast}\big( \frac{19}{10} \big)$ 
 in terms of the specified root of \eqref{20240203111868A8871M}, 
 as suggested in the Mathematica input below. 
\begin{verbatim}
a = 4165218 + 2945250*Sqrt[2] - 1862095*Sqrt[5] - 1316700*Sqrt[10];
L = Root[1 - 632783448*#1 + 12127295380*#1^2 - 632783448*#1^3 - 
24254590746*#1^4 + 632783448*#1^5 + 12127295380*#1^6 + 632783448*#1^7 + 
#1^8 & , 6, 0];
c = (L^(-1) - L)^2/16;
r = (-a + Sqrt[1 + a^2])^(1/5);
lambda = Sqrt[-c^(-1) + Sqrt[1 + 4*c]/c]/Sqrt[2];
MinimalPolynomial[-256*(-1 + lambda^2 - lambda^4)^3*r^5*(-1 + 11*r^5 + 
r^10)^5 + lambda^4*(-1 + lambda^2)^2*(1 + 228*r^5 + 494*r^10 - 228*r^15 + 
r^20)^3,x]
\end{verbatim}
 We may employ a similar line of reasoning as in our proof of Theorem  \ref{theoremg130}, by comparing the numerical values of  $R\big( e^{-\pi 
 \sqrt{ 38/5 }} \big)$ and  $\sqrt[5]{ \sqrt{a^2+1}-a}$ to the possible roots of the polynomial given by the icosahedral equation. 
\end{proof}

\section{A new closed form related to $G_{240}$}
 What is especially remarkable about Theorem \ref{20240203658AM1A} below is 
 given by how \emph{separate} features of the Maple and Mathematica computer algebra systems 
 are required in conjunction, 
 in the derivation of closed forms required to prove 
 Theorem \ref{20240203658AM1A}. Notably, Mathematica does not seem to be able to 
 compute the required minimal polynomial indicated in Lemma \ref{20240204610A2M2A}. 

 The evaluation 
\begin{equation}\label{20240203811AM1A}
 G_{15} = 2^{1/4} \left( \frac{1 + \sqrt{5}}{2} \right)^{1/5}
\end{equation}
 is given in \cite[p.\ 190]{Berndt1998}. 
 Writing $$ \mathcal{C}_{15} 
 = 4 G_{15}^8 \left(G_{15}^8 + \sqrt{-\frac{1}{G_{15}^8} + G_{15}^{16}} \right), $$
 this allows us to obtain a closed form for 
 $$ G_{60} = \frac{ \left(\mathcal{C}_{15} + 
 \frac{\sqrt{\mathcal{C}_{15}^3 + 8}}{\sqrt{\mathcal{C}_{15}}} \right)^{1/8}}{2^{1/4}}, $$
 according to the modular relation in \eqref{20240203851AM2A}. 
 By then writing 
 $$ \mathcal{C}_{60} = 4 G_{60}^8 \left(G_{60}^8 + \sqrt{-\frac{1}{G_{60}^8} + G_{60}^{16}} \right), $$
 we may obtain a closed form for 
\begin{equation}\label{2024702074537292A272M2A}
 G_{240} = \frac{ \left( \frac{ \mathcal{C}_{60} + 
 \sqrt{\mathcal{C}_{60}^3 + 8}}{\sqrt{\mathcal{C}_{60}}} \right)^{1/8}}{2^{1/4}}, 
\end{equation}
 again by the modular relation in \eqref{20240203851AM2A}. By then setting 
\begin{equation}\label{20240203624PM1A}
 \alpha = \frac{1}{2} \left(1 - \frac{ \sqrt{G_{240}^{24} - 1}}{G_{240}^{12}} \right) 
\end{equation}
 so as to agree with \eqref{alphafor25}, this leads us toward the following result. 

\begin{lemma}\label{20240204610A2M2A}
 The minimal polynomial for $\alpha$, as defined in \eqref{20240203624PM1A}, equals 
\begin{align*}
 & x^{16} - 85646102224053010448 x^{15} + 59547310292447325609394296 x^{14} - \\ 
 & 63252200262236651831473406512 x^{13} + 525839570761535689444949755676 x^{12} - \\
 & 1979219931663657931544660611344 x^{11} + 4551988046736278352673918558024 x^{10} - \\ 
 & 7226577193130665396845546777776 x^9 + 8382324320686645930076221747782 x^8 - \\ 
 & 7226577193130665396845546777776 x^7 + 4551988046736278352673918558024 x^6 - \\ 
 & 1979219931663657931544660611344 x^5 + 525839570761535689444949755676 x^4 - \\
 & 63252200262236651831473406512 x^3 + 59547310292447325609394296 x^2 - \\
 & 85646102224053010448 x + 1. 
\end{align*}
\end{lemma}

\begin{proof}
 Inputting 
\begin{verbatim}
G15closed := 2^(1/4)*((1 + sqrt(5))/2)^(1/3);
C15 := 4*G15closed^8*(G15closed^8 + sqrt(-1/G15closed^8 + G15closed^16));
G60closed := (C15 + sqrt(C15^3 + 8)/sqrt(C15))^(1/8)/2^(1/4);
C60 := 4*G60closed^8*(G60closed^8 + sqrt(-1/G60closed^8 + G60closed^16));
G240closed := (C60 + sqrt(C60^3 + 8)/sqrt(C60))^(1/8)/2^(1/4); 
alphap := 1/2*(1 - sqrt(G240closed^24 - 1)/G240closed^12);
evala(Minpoly(alphap, x))
\end{verbatim}
 into Maple, the associated output confirms the desired result. 
\end{proof}

\begin{theorem}\label{20240203658AM1A}
 The evaluation 
 $$ R^{5}\left( e^{-\pi \sqrt{\frac{48}{5}}} \right) = \sqrt{a^2 + 1} - a $$
 holds, where
 $$ a = 2118 + \frac{1885 \sqrt{5}}{2} + \frac{4875 \sqrt{3}}{4} + \frac{2175 \sqrt{15}}{4}. $$
\end{theorem}

\begin{proof}
 We claim that $\lambda^{\ast}\big( \frac{48}{5} \big)$ is equal to 
 the third root of 
\begin{align*}
 & 1 - 9254518800 x + 7997750214776 x^2 - 238623871222320 x^3 - \\
 & 395374840051940 x^4 + 722354076987120 x^5 + 1549442293997384 x^6 - \\
 & 413352207084720 x^7 - 2183392919932346 x^8 - 413352207084720 x^9 + \\
 & 1549442293997384 x^{10} + 722354076987120 x^{11} - 395374840051940 x^{12} - \\
 & 238623871222320 x^{13} + 7997750214776 x^{14} - 9254518800 x^{15} + x^{16}. 
\end{align*}
 To prove this evaluation for $\lambda^{\ast}$, our strategy 
 is to use the $G$-value in \eqref{20240203811AM1A} together with modular relations. 

 Let $\mathcal{L}$ denote the desired value for $\lambda^{\ast}\big( \frac{48}{5} \big)$. 
 By rewriting the $\beta$-value in \eqref{betafor25} 
 in terms of $\lambda^{\ast}\big( \frac{48}{5} \big)$, 
 the Mathematica output associated with the following input gives us that:
 By replacing $\lambda^{\ast}\big( \frac{48}{5} \big)$
 with the desired value $\mathcal{L}$, the desired modular relation of order 25 in 
 \eqref{modular25} is satisfied. By replacing $\lambda^{\ast}\big( \frac{48}{5} \big)$ 
 with a free variable and consider all possible numerical values 
 such that \eqref{modular25} holds, we can show that $\mathcal{L} = \lambda^{\ast}\big( \frac{48}{5} \big)$. 
 Lemma \ref{20240204610A2M2A} is being used to express $\alpha$ in the following input. 
\begin{verbatim}
L = Root[1 - 9254518800 # + 7997750214776 #^2 - 
238623871222320 #^3 - 395374840051940 #^4 + 722354076987120 #^5 + 
1549442293997384 #^6 - 413352207084720 #^7 - 2183392919932346 #^8 - 
413352207084720 #^9 + 1549442293997384 #^10 + 722354076987120 #^11 - 
395374840051940 #^12 - 238623871222320 #^13 + 7997750214776 #^14 - 
9254518800 #^15 + #^16& , 3, 0];
\[Alpha] = Root[#1^16 - 85646102224053010448*#1^15 + 
59547310292447325609394296*#1^14 - 63252200262236651831473406512*#1^13 + 
525839570761535689444949755676*#1^12 - 
1979219931663657931544660611344*#1^11 + 
4551988046736278352673918558024*#1^10 - 
7226577193130665396845546777776*#1^9 + 
8382324320686645930076221747782*#1^8 - 
7226577193130665396845546777776*#1^7 + 
4551988046736278352673918558024*#1^6 - 
1979219931663657931544660611344*#1^5 + 
525839570761535689444949755676*#1^4 - 63252200262236651831473406512*#1^3 + 
59547310292447325609394296*#1^2 - 85646102224053010448*#1 + 1 &, 1, 0];
\[Beta] = 1/2 (1 - 2 Sqrt[L^2 - L^4] Sqrt[1/(4 L^2 - 4 L^4) - 1]);
P = (16 \[Alpha] \[Beta] (1 - \[Alpha]) (1 - \[Beta]))^(1/12);
Q = ((\[Beta] (1 - \[Beta]))/(\[Alpha] (1 - \[Alpha])) )^(1/8);
MinimalPolynomial[(Q + 1/Q) - (-2 (P - 1/P)), x]
\end{verbatim}
 In the icosahedral equation, we set $\tau = i \sqrt{\frac{12}{5}}$. 
 We then rewrite $\lambda\big( i \sqrt{\frac{12}{5}} \big)$ 
 as $\big(\lambda^{\ast}\big( \frac{12}{5} \big)\big)^{2}$. 
 The modular equation in \eqref{20240203851AM2A}
 then allows us to rewrite $\lambda^{\ast}\big( \frac{12}{5} \big)$ in terms 
 of $\lambda^{\ast}\big( \frac{48}{5} \big)$. 
 Inputting the following into Mathematica, 
 for the same value of {\tt L} indicated above, 
 this shows us that the desired value for 
 $R\big( e^{-\pi \sqrt{48/5}} \big)$ satisfies the 
 icosahedral equation for the corresponding values of $\lambda^{\ast}$. 
\begin{verbatim}
c = 1/4 (1/L - L)^2;
a = 2118 + (1885 Sqrt[5])/2 + (4875 Sqrt[3])/4 + (2175 Sqrt[15])/4;
r = ((Sqrt[a^2 + 1] - a)^(1/5));
MinimalPolynomial[64 (4 + c)^3 r^5 (-1 + 11 r^5 + r^10)^5 - (-c^2 (1 + 
228 r^5 + 494 r^10 - 228 r^15 + r^20)^3), x]
\end{verbatim}
 Again, we may employ a similar line of reasoning as in our proof of Theorem 
 \ref{theoremg130}, by comparing the numerical values of 
 $R\big( e^{-\pi \sqrt{ 48/5 }} \big)$ and 
 $\sqrt[5]{ \sqrt{a^2+1}-a}$ to the possible roots of the 
 polynomial given by the icosahedral equation. 
\end{proof}

\section{A conjectured evaluation}
 We briefly conclude with the following experimentally discovered closed form. 
 The value $\frac{16}{15}$ involved in Conjecture \ref{20240204537} 
 presents an obstacle in the following sense. 
 In the hope of applying the closed form for $G_{240}$ 
 given indirectly in \eqref{2024702074537292A272M2A}, 
 we would require a combination of applications of 
 modular relations of order 9 and of order 25. 
 This proves to be problematic, in the sense that 
 we would not be able to express the required $\alpha$-value 
 explicitly, as in \eqref{20240203624PM1A}. 

\begin{conjecture}\label{20240204537}
 The evaluation 
 $$ R^{5}\left( e^{-\pi \sqrt{\frac{16}{15}}} \right) 
 = \sqrt{a^2 + 1} - a $$
 holds, where
 $$ a = \frac{1}{4} \left(8472+4875 \sqrt{3}-3770 \sqrt{5}-2175 \sqrt{15}\right) $$
\end{conjecture}

\subsection*{Acknowledgements}
 The author is grateful to acknowledge support from a Killam Postdoctoral Fellowship from the Killam Trusts. 
 The author wants to thank Karl Dilcher, Shane Chern, and Lin Jiu 
 for useful comments related to this paper.

 \

John M.\ Campbell

{\tt jmaxwellcampbell@gmail.com}

Department of Mathematics and Statistics

Dalhousie University

\end{document}